\documentclass{amsart}
\usepackage[utf8]{inputenc}
\usepackage{amsmath}
\usepackage{amssymb}
\usepackage{amsfonts}
\usepackage{amsthm}
\usepackage[pdftex]{hyperref}
\usepackage{geometry}
\usepackage{enumerate,amssymb}
\usepackage[arrow,curve,matrix,tips,2cell]{xy}

\UseAllTwocells
\usepackage{mathtools}
\usepackage{amscd}
\usepackage[mathscr]{euscript}
\newcommand{\K}{\mathcal{K}}
\newcommand{\R}{\mathbb{R}}

\newcommand{\integers}{\mathbb{Z}}
\newcommand{\rationals}{\mathbb{Q}}
\newcommand{\boundary}{\partial}
\newcommand{\disjointunion}{\amalg}
\newcommand{\abs}[1]{\left\lvert#1\right\rvert} 
\theoremstyle{plain}
\newtheorem{theorem}{Theorem}[section]
\newtheorem{lemma}[theorem]{Lemma}

\theoremstyle{definition}
\newtheorem{definition}[theorem]{Definition}
\newtheorem{example}[theorem]{Example}

\newtheorem{remark}[theorem]{Remark}

\title{Conormal homology of manifolds with corners}
\author{Thomas Schick}
\email{\href{mailto:thomas.schick@math.uni-goettingen.de}{e-mail:
  thomas.schick@math.uni-goettingen.de}}
\urladdr{\href{http://www.uni-math.gwdg.de/schick}{www:~http://www.uni-math.gwdg.de/schick}}
\address{Mathematisches Institut\\
Universit\"at G{\"o}ttingen\\
Germany}

\author{Mario Vel\'asquez}
\thanks{This work was carried out during a research stay of the second author at
  Univerity of G\"ottingen funded by the German Science Foundation RTG ``Fourier analysis and spectral theory''}
  
\email{\href{mailto:mavelasquezme@unal.edu.co}{e-mail:mavelasquezme@unal.edu.co}}
\urladdr{\href{https://sites.google.com/site/mavelasquezm/}{www:~https://sites.google.com/site/mavelasquezm/}}\address{Current addresss: Departamento de Matem\'aticas\\Universidad Nacional de Colombia\\Cra. 30 cll 45 - Ciudad Universitaria\\ Bogot\'a, Colombia}
\date{November 2021}

\begin{document}
	\maketitle

\begin{abstract}
   Given a manifold with corners $X$, we associates to it the
   corner structure simplicial complex $\Sigma_X$. Its reduced K-homology is isomorphic
   to the K-theory of the $C^*$-algebra $\K_b(X)$ of b-compact operators on
   $X$. Moreover, the homology of $\Sigma_X$ is isomorphic to the conormal
   homology of $X$.


   In this note, we construct for an arbitrary abstract finite simplicial complex
   $\Sigma$ a manifold with corners $X$ such that $\Sigma_X\cong \Sigma$. As a
   consequence, the homology and K-homology which occur for finite simplicial
   complexes also occur as conormal homology of manifolds with corners and as
   K-theory of their b-compact operators. In particular, these groups can
   contain torsion.
\end{abstract}

\section{Introduction}
In this note we contribute to the index theory and homology of compact manifolds with
corners. More specifically, a fundamental question in this field asks for
obstructions to the Fredholm property for boundary value problems, initiated
in \cite{CL}. Let $X$ be such a manifold with corners. The geometrically
relevant boundary value problems (and their inverses) are contained in the algebra of
b-pseudodifferential operators as introduced by \cite{MelrosePiazza}. Given
such an operator which is \emph{elliptic} (has an elliptic principal symbol)
it is in general not true that the operator is Fredholm.

In this situation, on asks if one can find a smoothing perturbation (possibly after
stabilization) to render the given operator Fredholm. If this is possible, the
operator satisfies the stable Fredholm perturbation property (SFP). As one
can guess, not all operators have the SFP: it is proven in \cite{NSS} that the obstruction to this is precisely
the
boundary analytic index which takes values in $K_*(\K_b(\boundary X))$ where
$\K_b(\boundary X)$ is the $C^*$-algebra of b-compact operators on the
boundary of $X$ (the $C^*$-algebra $\K_b(\boundary X)$ is defined as the quotient $\K_b(X)/\K(X)$, where $\K(X)$ denotes the $C^*$-algebra of compact operators on $L^2(X)$, for details see \cite{CL}). A modern proof of the same fact, using deformation groupoids, is given in
\cite{CLM}. By \cite[Proposition 5.6]{CL}, the restriction to the
boundary induces an isomorphism $K_*(\K_b(X))\to K_*(\K_b(\boundary X))$
provided $\boundary X\ne \emptyset$.

It turns out that the relevant group $K_*(\K_b(X))$ depends only on the
combinatorics of the faces of the manifold with corners $X$ and how they
intersect. Indeed, one essentially can compute these K-groups
combinatorially. To this end, Bunke \cite{Bunke}  introduced the
\emph{conormal homology} of $X$ (called differently in \cite{Bunke}), computed from a chain complex generated
abstractly by the faces of $X$. In \cite{CL}, it is shown that this conormal
homology contains the obstructions to SFP if the corners in $X$ are of
codimension $\le 3$.

More systematically, in \cite{CLV} a natural Chern character
\begin{equation}\label{eq:Chern}
  K_*(\K_b(X))\to H^{cn}_{*+2\integers}(X)\otimes\rationals
\end{equation}
is constructed and proven to be rationally an isomorphism.

In the situation studied in \cite{CL} it is useful if the conormal homology is
torsion free. There and in \cite{CLV}, the authors therefore ask if the
conormal homology always is torsion free, and whether perhaps the Chern
character \eqref{eq:Chern} can be improved to an integral isomorphism.

The main goal of the present paper is to provide examples of manifolds with
corners which show that the conormal homology does not have any of these nice
properties.
Instead, it can be as rich as the homology of an arbitrary finite simplicial
complex. The same applies to the K-theory of $\K_b(X)$.

More specifically, in the present paper, we associate to a  manifold with
embedded corners $X$ its \emph{corner structure complex}, a simplicial complex
$\Sigma_X$ encoding the \emph{corner structure} in such way that the reduced
K-homology of $\Sigma_X$ is isomorphic to the topological K-theory of the
$C^*$-algebra  $\K_b(X)$, and the conormal homology of $X$ isomorphic to the
homology of $\Sigma_X$. We then prove that every finite simplicial complex can
be realized as $\Sigma_X$ of some manifold with embedded corners $X$.

In particular it implies that the conormal homology defined in \cite{CL} and
\cite{CLV} in general does contain torsion.

\section{The corner structure complex of a manifold with corners}

	Let us recall the definition of a smooth manifold with embedded corners. We
        adopt the approach where a manifold with corners is defined as a
        suitable subset of a smooth ordinary manifold (without boundary).

\begin{definition}\label{def:mf_with_corners}
          A compact smooth manifold with embedded corners $X$ is defined in the following
          way. Start with a compact smooth manifold $\widetilde{X}$ (without
          boundary) and with smooth maps $\rho_0,\ldots,\rho_n\colon
          \widetilde{X}\to\mathbb{R}$. Set
    \begin{equation*}
            H_j=\rho_j^{-1}(0)\cap X,\; j=0,\ldots,n, \quad H_{j_1,\dots,j_k}:= H_{j_1}\cap\ldots\cap H_{j_k}.
          \end{equation*}
           such that
          \begin{equation*}
            \{d\rho_{j_1},\ldots,d\rho_{j_k}\} \text{ has maximum rank at each }
            x\in H_{j_1,\dots,j_k}\quad\forall \{j_1,\dots,j_k\}\subset\{0,\dots,n\}
          \end{equation*}
     This defines the manifold with corners $X$ as
		\begin{enumerate}
			\item $X:=\bigcap_{j=1}^n\rho_j^{-1}([0,+\infty))\subseteq\widetilde{X}$.
			\item Each $H_j$ is called a \emph{boundary component
                            of codimension 1}.
                        \item  If $\abs{\{j_1,\dots,j_k\}}=k$ then we call
                          $H_{j_1}\cap\ldots\cap H_{j_k}$ a \emph{face of
                            codimension $k$}.
                        \item We  denote by $F_k$ the set of faces of codimension $k$.
		\end{enumerate}

  Throughout the paper, we assume that  each face $H_{j_1}\cap\ldots\cap
  H_{j_k}$ of arbitrary codimension is connected, in particular each boundary component of codimension 1.
\end{definition}

\begin{definition}\label{simplicial-complex}
		Let $X$ be a manifold with embedded corners with $n+1$ boundary
                components $H_0$,\dots, $H_n$ of codimension $1$. Define the \emph{corner
                  structure complex} as the abstract simplicial complex $\Sigma_X$
                associated to $X$ with vertex set $\{H_0,\ldots,H_n\}$ with
                the following simplices: for $A\subseteq\{H_0,\ldots,H_n\}$
		\begin{center}
			$A\in \Sigma_X$ if and only if $\bigcap_{i\in A}H_i\neq\emptyset$.
		\end{center}
	\end{definition}
	It is clear that $\Sigma_X$ is closed under inclusions and therefore
        is an abstract simplicial complex. Observe that an abstract simplicial complex as considered here is a set, the vertex set (here $\{H_0,\dots,H_n\}$), together with a collection of finite subsets (here $\Sigma_X)$, called the simplices, which is closed under taking subsets (if $\tau\subset \sigma$ then $\tau$ is called a face of $\sigma$).
        
Every abstract simplicial complex $\Sigma$  has a \emph{geometric realization} denoted by $|\Sigma|:= (\bigcup_{\sigma\in\Sigma} \sigma\times \Delta^{|\sigma|-1})/\sim$. Here, $\Delta^k$ is the standard $k$-simplex and we glue according to the face relation in the abstract simplicial complex. For details compare \eqref{eq:simpl_complex} or consult \cite[Theorem. 7.8]{rotman}.

\begin{example}\label{ejem}
		\begin{enumerate}
			\item If $X$ be a smooth manifold without boundary then $\Sigma_X=\emptyset$.
			\item Let $(Y,\partial Y)$ be a connected manifold
                          with non-empty connected boundary. Then $\Sigma_Y$ is a
                          point.
                          
            \item Let $X$ be a connected manifold with boundary with $n$ boundary
  components. Then $\Sigma_X$ is the simplicial complex with $n$ points and no edges.
  \end{enumerate}
                          
\end{example}
More generally we have the following result
\begin{lemma}
  Let $Y_1$ and $Y_2$ be manifolds with embedded corners, then $\Sigma_{Y_1\times Y_2}$ is isomorphic to the join  $\Sigma_{Y_1}*\Sigma_{Y_2}$.
\end{lemma}
\begin{proof}
Recall that the join of two abstract simplicial complexes
    $\Sigma_1,\Sigma_2$ has vertex set the disjoint union
    $V(\Sigma_1)\disjointunion V(\Sigma_2)$, where $V(\Sigma_j)$ is the vertex
    set of $\Sigma_j$ for $j=1,2$. A subset $A\subset V(\Sigma_1)\disjointunion
    V(\Sigma_2)$ is a simplex of $\Sigma_1*\Sigma_2$ if and only if
    $A=A_1\disjointunion A_2$ with $A_1\in \Sigma_1$ and $A_2\in\Sigma_2$.
    
  Let $\widetilde{Y}_1$ be a smooth manifold and let
  $\rho_1,\ldots,\rho_n\colon \widetilde{Y}_1\to\mathbb{R}$ be smooth maps defining the
  manifold with corners $Y_1$, In the same way, let $\widetilde{Y}_2$ and
  $\rho_{n+1},\ldots,\rho_{n+m}
 \colon \widetilde{Y}_2\to \mathbb{R}$ define
  $Y_2$. Then the smooth manifold $\widetilde{Y}_1\times\widetilde{Y}_2$ with
  the smooth maps $\rho_1\circ \pi_1,\ldots,\rho_n\circ\pi_1, \rho_{n+1}
  \circ \pi_2,\ldots,\rho_{n+m}
  \circ \pi_2
  \colon \widetilde{Y}_1\times\widetilde{Y}_2\to\mathbb{R}$ define $Y_1\times
  Y_2$ as a manifold with embedded corners. Here $\pi_j\colon Y_1\times Y_2\to
  Y_j$, $j=1,2$, are the projections. We denote the boundary components of $Y_1$   by $H_j$, $j=1,\ldots,n$ and the boundary components of  $Y_2$ by $H_j$, $j=n+1,\ldots,n+m$. Then the set of boundary components of $Y_1\times Y_2$ is $\{H_1\times Y_2,\ldots, H_n\times Y_2, Y_1\times H_{n+1},\dots,Y_1\times H_{n+m}\}$ and clearly $\Sigma_{Y_1\times Y_2}$ satisfies the conditions to be the join $\Sigma_{Y_1}*\Sigma_{Y_2}$.
\end{proof}

 Note that the geometric realisation of the join of simplicial complexes is
    the topological join of the geometric realisations of the individual
    simplicial complexes.
The above lemma gives more examples:
\begin{example}
\begin{enumerate}
    \item Let $(Y,\partial Y)$ be a connected manifold with non-empty connected  boundary and set $X=Y^n$. This is a manifold with embedded corners and  $\Sigma_X=\Delta_n$ is the $n$-simplex.
\item\label{item:low_d_ex} We can also directly construct an $(n+2)$-dimensional manifold $A$ with
  embedded corners (all of whose faces are connected) such that $\Sigma_A\cong
  \Delta_n$. For this aim, start with $\tilde A:=S^{n+2}$, but decompose $$S^{n+2}=\partial(D^{n+3})=\partial(D^2\times D^{n+1})=S^1\times D^{n+1}\cup_{S^1\times S^n}D^2\times S^n.$$ This way,
  we have an obvious projection map to the second factor $\pi\colon
  S^{n+2}\to D^{n+1}$, with $n+1$ component functions
  $\rho_1,\dots,\rho_{n+1}\colon \tilde A\to \R$. Then we set
  \begin{equation*}
    A:=\bigcup_{j=1}^{n+1} \{\rho_j\ge 0\} = S^1\times Q^{n+1}\cup_{S^1\times
    (S^n\cap Q^{n+1})} D^2\times (S^n\cap Q^{n+1}),
  \end{equation*}
  where we write $Q^{n+1}:=\{ (x_1,\dots,x_{n+1})\in D^{n+1}\mid x_j\ge
  0, \forall j\}$, the positive hyperquadrant sector. Then $A$ is a manifold
  with embedded corners where clearly all the faces are connected  and as $\bigcap_{j=1}^{n+1}\{\rho_j=0\}=S^1\times\{(0,\ldots,0)\}\neq\emptyset$, according to Definition \ref{simplicial-complex} we have $\Sigma_A\cong \Delta_n$, as desired.

\end{enumerate}

\end{example}

Our first main result is constructive: for every finite simplicial complex we can
construct a manifold with corners. A more general result will be proved in Theorem \ref{realization}.

\begin{theorem}\label{theo:main_constr}
  Let $\Sigma$ be an arbitrary finite simplicial complex. Then there exists a
  compact manifold with embedded corners $X$ such that $\Sigma_X$ is isomorphic
  to $\Sigma$.
\end{theorem}

  As a preparation for the proof, we  introduce further definitions.

\begin{definition}
  Let $X$ be a manifold with embedded corners and let $x_0\in X$,
  let $V_0$ be a coordinate neighborhood around $x_0$. Then $V_0$ is a
  manifold with embedded corners and corner structure complex $\Sigma_{V_0}$ is called the simplicial complex of $X$ around $x_0$. Given points $x_0\in X$ and $y_0\in Y$, where $X$ and $Y$ are manifolds with embedded corners of the same dimension, we say that \emph{$x_0$ and $y_0$ have the same local corner structure} if the corner structure complexes of $X$ around $x_0$ and of $Y$ around $y_0$ are isomorphic.  
\end{definition}

 Now we will define the connected sum around 0-dimensional submanifolds. This will be
 a key ingredient in the constructions required to prove Theorem \ref{theo:main_constr}.
	
  Let $X$ and $Y$ be manifolds with embedded corners of dimension $n$, let $\{x_1,\ldots, x_m\}\subseteq X$ and $\{y_1,\ldots, y_m\}\subseteq Y$ be finite subsets of $X$ and $Y$, respectively, such that for $i=1,\ldots m$, $x_i$ and $y_i$ have the same local corner structure. 
By results in \cite{MelroseBook} there are tubular neighborhoods $V$ and $W$ with
	$$\{x_1,\ldots,x_m\}\subseteq V\subseteq X\text{ and } \{y_1,\ldots,y_m\}\subseteq W\subseteq Y.$$
	We have
	$$V\cong V_1\sqcup\ldots\sqcup V_m \text{ and }W\cong
        W_1\sqcup\ldots\sqcup W_m,$$
        where $V_i$ is a coordinate neighborhood
        of $x_i$ and $W_i$ is a coordinate neighborhood of $y_i$. Then we can
        choose identifications $$V_i= (-1,1)^{n_i}\times[0,1)^{n-n_i}\subset
        (-1,1)^{n_i}\times (-1,1)^{n-n_i}=: \widetilde{V_i},$$
     and
	$$W_i= (-1,1)^{m_i}\times[0,1)^{n-m_i}\subset\widetilde{W_i}.$$
In these coordinates, the face defining functions $\rho_j$ of Definition
\ref{def:mf_with_corners} are just the coordinate functions for the closed intervals.

Moreover, as $x_i$ and $y_i$ have the same local corner structure, we have $n_i=m_i$.
	
We now follow the description of the ordinary connected sum in \cite{Milnorcs}
to define our connected sum of manifolds with corners. 

\begin{definition}
  Define the \emph{connected sum of $X$ and $Y$ along the subsets
    $\{x_0,\ldots,x_m\}$ and $\{y_0,\ldots,y_m\}$} as follows:
  \begin{equation}
	(X,\{x_0,\ldots,x_m\})\sharp
        (Y,\{y_0,\ldots,y_m\}):=\frac{(X-\{x_0,\ldots,x_m\})\bigsqcup(Y-\{y_0,\ldots,y_m\})}{tz\in
          V_i\sim(1-t)z\in W_i }\label{eq:connect_sum}
\end{equation}
	for every $z\in S^{n-1}$.  When $\{x_0,\ldots,x_m\}$ and
        $\{y_0,\ldots,y_m\}$ are clear from the context, we denote the
        connected sum by $X\sharp Y$.
\end{definition}

  For the sake of completeness, let us give the details how it is an
  straightforward verification that $X\sharp Y$ is a manifold with
  embedded corners. To do so, we have to define the ambient smooth ordinary
  manifold and the boundary defining functions.

  Of course, the ambient manifold for the connected sum is defined as the
  connected sum of the ambient manifolds $\tilde X$ of $X$ and $\tilde Y$ of
  $Y$, defined precisely with the same formula as \eqref{eq:connect_sum}.

  The boundary defining function $\rho_j$ remains unchanged outside the set
  $V_i\sqcup W_i$ but have to be modified in the
  coordinate regions where the identifications are carried out. We redefine $\rho_j$ on
$V_i\sqcup W_i$ by
\begin{equation*}
  \rho_j(x):=\phi_x(\abs{x})\cdot x_j;  
\end{equation*}
with and $\phi_x\colon (0,1)\to (0,1)$ a smooth monotonously decreasing function with
\begin{equation*}
  \phi_x(t) =
  \begin{cases}
    1; & t> 0.6\\
    \frac{1-t}{t}; & t<0.4
  \end{cases}
\end{equation*}
In the coordinates of $W_i$, we have for $y\in W_i$ that $y\sim x:=
\frac{1-\abs{y}}{\abs{y}} y\in V_i$ and therefore
\begin{equation*}
  \rho_j(y) =
  \begin{cases}
    \frac{1-\abs{y}}{\abs{y}} y_j; & 1-\abs{y}>0.6 \equiv \abs{y}<0.4\\
    \frac{1-(1-\abs{y})}{\abs{y}} \frac{1-\abs{y}}{\abs{y}} y_j=y_j & \abs{y}>0.6
  \end{cases}
  \end{equation*}
We observe that we indeed defined a smooth function on the connected sum of
$\tilde X$ and $\tilde Y$ as the expressions are unchanged on the ``outer''
part of $U_i$ or $V_i$, where the norm is $>0.6$.

\begin{remark}
  The connected sum of two manifolds with corners $X$ and $Y$  will again be a manifold
  with corners. However, even if all faces of arbitrary codimension of $X$ and
  of $Y$ are connected, in general this will not be the case for $X\sharp Y$.
\end{remark}

\begin{lemma}\label{pasting lemma}
  Assume that $X$ is a compact manifold of dimension $d$ with embedded corners
  (all faces connected). Assume we have an embedding $\iota \colon \boundary
  \Delta_n\hookrightarrow \Sigma_X$, but the simplex spanned by the vertices in the
  image of $\iota$ is not contained in $\Sigma_X$.

  Then for each $d'\ge \max\{d, n+2\}$ there is a compact smooth manifold $Z$
  with embedded corners (all faces connected) such that $\Sigma_Z\cong
  \Sigma_X\cup_{\iota}\Delta_{n}$, i.e.~$\Sigma_Z$ is isomorphic to the
  simplicial complex obtained from $\Sigma_X$ by adding the simplex spanned by
  the image of $\iota$.
\end{lemma}
\begin{proof}
  We write $\Delta_n=\mathcal{P}(\{0,\dots,n\})$, the power set of $\{0,\dots,n\}$.
  Taking the product with a smooth connected manifold, a process which does
  not change the corner structure, we can assume that $\dim(X)=d'$. Let
  $H_k\subset X$ be the boundary component which corresponds to $\iota(k)$
  for $k=0,\dots,n$, and set 
  \begin{equation*}
    H_{k^c}:= \bigcap_{j\ne k} H_j,
  \end{equation*}
  the face of $X$ of codimension $n$ corresponding to
  $\{0,\dots,n\}\setminus\{k\}\subset \boundary\Delta_n$. By the assumption
  that $\iota(\boundary\Delta_n)\subset \Sigma_X$ we have $H_{k^c}\ne
  \emptyset$ and $\dim(H_{k^c})=d'-(n-1)>0$ for all $k$. Pick then $x_k\in
  (H_{k^c})^\circ$. The local corner structure around $x_k$ is precisely
  the $(n-1)$-simplex spanned by $\{0,\dots,n\}\setminus\{k\}$.

Choose a $d'$-dimensional manifold $M$ with embedded corners (all faces
connected) with a fixed isomorphism $\Delta_n\xrightarrow{\cong}\Sigma_M$. For this, start with $A$ of
Example \ref{ejem} (\ref{item:low_d_ex}) and take the Cartesian product with a
connected smooth manifold to adjust the dimension. Given the fixed isomorphism
$\Sigma_M\cong\Delta_n$, let $H_k^M$ be the boundary face corresponding to the
vertex $k$ of $\Delta_n$ and set $H^M_{k^c}:=\bigcap_{j\ne k} H^M_j$. Note
that $\dim(H^M_{k^c})=d'-n>0$ and pick $y_k\in (H_{k^c}^M)^\circ$, with local
corner structure the $(n-1)$-simplex spanned by $\{0,\dots,n\}\setminus
\{k\}$. 

Define now $Z:= (X; x_0,\dots,x_n)\# (M;y_0,\dots,y_n)$, using at each point
the given identifications of the local corner structure with a determined
$(n-1)$-dimensional subsimplex of $\Delta_n=\mathcal{P}(\{0,\dots,n\})$. This
way, the face $H_k$ is glued near the $n$ points $x_j$ with
$j\in\{0,\dots,n\}\setminus\{k\}$ to the face $H_k^M$. Moreover, for each
$0\le k_0<\dots<k_\alpha\le n$ the face
$H_{k_0}\cap\dots H_{k_\alpha}$ is glued (at $n-\alpha$ points) to the
intersection face $H^M_{k_0}\cap\dots \cap H^M_{k_\alpha}$. It follows that
each face of arbitrary codimension of $Z$ remains connected. The boundary
faces of $Z$ are in obvious bijection with those of $X$ and the corner complex
is a simplicial complex with vertex set the one of $\Sigma_X$. As
$H_0^M\cap\dots H_n^M\ne \emptyset$, taking everything together, we now get
$\Sigma_Z=\Sigma_X\cup_{\iota(\boundary\Delta_n)}\iota(\Delta_n)$. 
  \end{proof}

\begin{theorem}\label{realization}
	Let $K$ be a finite simplicial complex and let $n$ be the maximal
        dimension of a simplex in $K$. If $d\ge n+2$ then there is a compact
        $d$-dimensional manifold with embedded corners $X$ (all faces
        connected) such that $$\Sigma_X\cong K.$$ 
\end{theorem}
\begin{proof}
	The result follows by induction on the number of positive dimensional
        simplices in $K$.
	
	By Example \ref{ejem} (3) any 0-dimensional simplicial complex
        satisfies the result.

        If $K$ contains a positive dimensional simplex, let $\sigma$ be one
        such of maximal dimension. Then $K=K'\cup_{\boundary\sigma}
        \sigma$. By induction, there exists the required $X'$ with $\Sigma_{X'}\cong K'$
        and by Lemma \ref{pasting lemma} we then can also construct the
        required $X$ with $\Sigma_X\cong K$.
\end{proof}
      
\section{Groupoids and the space $O_X$}
Let $X$ be a manifold with embedded corners. In \cite{CLV}, for sufficiently
large $m$ a non-compact
topological space $O_X$ is introduced such that we have a Connes-Thom
isomorphism 

$$CT\colon K_*(C^*(\K_b(X)))\xrightarrow{\cong} K^{m+*}(O_X).$$

The space $O_X$ is constructed as the orbit space of a free and proper groupoid. We recall this construction briefly.

Let $X$ be defined by the smooth manifold $\widetilde{X}$ and the defining
functions $\rho_1,\dots,\rho_n$, i.e.
$$\bigcap_{j=1}^n \{\rho_j\ge 0\} =: X\subseteq\widetilde{X}.$$

The puff groupoid is defined as a subgroupoid of $\tilde X\times\tilde
X\times\R^n$, where $\tilde X\times \tilde X$ is the arrow space of the
pair groupoid and where $\R^n$ is the additive group (groupoid with one
object). The arrow space of the puff groupoid is then defined as
$$G(\widetilde{X},(\rho_i)):=\{(x,y,\lambda_1,\ldots,\lambda_n)\in
\widetilde{X}\times\widetilde{X}\times\R^n\mid
\rho_i(x)=e^{\lambda_i}\rho_i(y)\}.$$
Denote by $G_c(\widetilde{X},(\rho_i))$ the s-connected component of $G(\widetilde{X},(\rho_i))$.

Choose for sufficiently large an embedding $$\iota\colon
\widetilde{X}\to\R^{m-n}$$.

Using this embedding, in \cite[Section 3]{CLV} the authors construct a new groupoid
$$\R^m\rtimes_\iota G_c(\widetilde{X},(\rho_i)_i)\rightrightarrows\R^m\times
\widetilde{X}.$$
By \cite[Proposition 3.1]{CLV} this groupoid is a free and proper Lie groupoid
(this uses that all faces of each codimension are connected). Hence the orbit
space
$$O_{\widetilde{X}}=Orb(G_c(\widetilde{X},(\rho_i)_i)\rtimes_h\R^M) =
\R^m\times \tilde X/\sim$$
has a natural structure of smooth manifold. Decomposing $v\in\R^m$ as $v=(v',v'')\in\R^{m-n}\times\R^n$ set
\begin{align*}
\widetilde{\rho}_i\colon O_{\widetilde{X}}&\to\R\\
[(v,x)]&\to\rho_i(x)e^{v_i''}\notag.
\end{align*}
By \cite[Section 3]{CLV}, these maps are indeed well defined and determine a
manifold with corners $O_X$:
\begin{definition}
	We denote by $$O_X=\bigcap_{i=1}^n\{\widetilde{\rho}_i\geq0\}.$$
\end{definition}
In \cite{CLV} is verified that $O_X$ is manifold with embedded corners with defining functions $\widetilde{\rho}_1,\ldots,\widetilde{\rho}_n$.

The main result in \cite{CLV} is the following.
\begin{theorem}\label{CLV}
	There is a Connes-Thom isomorphism
	$$CT_h\colon K_*(C^*(\K_b(X)))\to K^{m+*}(O_X).$$
        Here $K^*(O_X)$ denotes topological K-theory with compact support.
      \end{theorem}
      
\section{Relation between $\Sigma_X$ and $O_X$}
In \cite[Section 3C]{CLV} the authors construct a filtration of the space $O_X$,
$$Y_0\subseteq Y_1\subseteq\ldots\subseteq Y_m=O_X,$$
moreover, in \cite[Proposition 3.27]{CLV} they prove for each $q$ that 
$$Y_q\setminus Y_{q-1}\cong\bigcup_{\stackrel{f\in \text{faces of }X}{ \  N-q \le \mathrm{codim}(f)\le N}}\R^{m}_f.$$ Here $\R^m_f$ is certain subspace of $\R^m$, and the construction implies that there is a homeomorphism from $O_X$  to a subspace of $\R^m$ which is entirely determined by the combinatorics of
the corner structure of $X$. We will now describe in detail this homeomorphism.

Canonically, $\tilde O_X$ is defined as subset of $\R^{m-|V|}\times
\R^{V}$ where $V$ is the set of boundary faces of $X$, i.e.~the vertex set
of $\Sigma_X$. We denote the standard basis
  vectors $\delta_H\in(\R_{\ge 0})^{V}$ for $H\in V$.
 Note that then $|\Delta(V)|\subset \R^V$ is the convex hull of the
 $\delta_H$, $H\in V$.

To define $\tilde O_X$, for every face $F$ of $X$ set
\begin{equation*}
  B_F:=\{x\in \R^{V}\mid x_H=0\text{ if }F\subset H, x_H>0\text{ if }F\cap
  H=\emptyset\iff F\nsubseteq H\}.
\end{equation*}
Recall that we have a correspondence between the faces $F$ and the subsets
$\sigma\in \Sigma_X\subset\mathcal{P}(V)$, namely $F=\bigcap_{H\in \sigma}
H$, and $B_F$ is the open positive quadrant ``spanned'' by all $\delta_H$ with
$H\notin \sigma$.

Define then
\begin{equation*}
  \tilde O_X:=\R^{m-|V|}\times \bigcup_{F\text{ face of }X} B_F\subset
  \R^{m-|V|}\times \R^V
\end{equation*}
equipped with the subspace topology. Note that, ass sets, the union is disjoint.
The computation in \cite[Section 3]{CLV} then gives a homeomorphism between
$O_X$ and $\tilde O_X$.

Note that $\Sigma_X=\Delta(V)$ is the full simplex spanned by $V$ if and only
if $\tilde O_X=\R^{m-|V|}\times (\R_{\ge 0})^{V}$. Otherwise, $0\notin \tilde
O_X$ and we have a homeomorphism
\begin{equation*}
  \tilde O_X \cong \R^{m-|V|}\times \R_{>0} \times P_X
\end{equation*}
where the factor $\R_{>0}$ is the norm (radial variable) and
\begin{equation*}
  P_X = |\Delta(V)|\cap \tilde O_X\subset (\R_{\ge 0})^{V}.
\end{equation*}

  Observe that, if $F=\bigcap_{H\in \sigma} H$ then $B_F\cap|\Delta(V)|$ is
  the interior of the convex hull $F_{V\setminus\sigma}$ of the $\delta_H$ for $H\notin \sigma$:
  \begin{equation*}
    B_F\cap|\Delta(V)| = (F_{V\setminus\sigma})^\circ.
  \end{equation*}

  Now
  \begin{equation*}
    |\Delta(V)|^\circ \subset P_X\subset |\Delta(V)|,
  \end{equation*}
i.e.~$|\Delta(V)|$ is a compactification of $P_X$. This implies that we get
for the one-point compactification $(P_X)^+$
\begin{equation*}
  (P_X)^+ \cong |\Delta(V)|/(|\Delta(V)|\setminus P_X),
\end{equation*}
i.e.~we identify all the missing points in the compactification $|\Delta(V)|$
of $P_X$ to one point.

We next show that $|\Delta(V)|\setminus P_X$ is the geometric realization of
the dual of the simplicial complex $\Sigma_X$. First recall that the geometric
realization of a simplicial complex $K$ with vertex set $V$ can be defined as follows:

If $I\subseteq V$,  set
$$F_I:= \text{closed convex hull of }\delta_V,\;v\in H .$$
Then
\begin{equation}\label{eq:simpl_complex}
    |K|=\bigcup_{I\in K}F_I = \disjointunion_{I\in K} (F_K)^\circ \subset \R^K.
\end{equation}
Here we define the open face $F_I^\circ:=F_I\setminus \boundary F_I$, in
particular $F_{\{H\}}=\{\delta_H\}$ is the singleton (not the empty set).

Let us now also recall the definition of the dual of a simplicial complex.
\begin{definition}
	Let $K$ be a simplicial complex with vertex set $V$ . Then the
        (Alexander) dual of $K$ is
        $$K^\vee=\{A\subset V\mid V-A\notin K\}.$$ 
\end{definition}

Note that, in particular,
\begin{equation}\label{eq:simplex_full}
  |\Delta(V)| = \disjointunion_{\sigma\subset V} (F_\sigma)^\circ.
\end{equation}

\begin{lemma}
  We have
  \begin{equation*}
    |\Sigma_X^\vee| =|\Delta(V)|\setminus P_X.
  \end{equation*}
\end{lemma}
\begin{proof}
  By definition,
  \begin{equation}\label{eq:dual_complex}
    |\Sigma_X^\vee|=\disjointunion_{\sigma\in\Sigma_X^\vee}(F_{\sigma})^\circ
                   =\disjointunion_{V\setminus\sigma\notin \Sigma_X}
                     (F_{\sigma})^\circ
  \end{equation}
  On the other hand, as observed above,
  \begin{equation}\label{eq:PX}
    P_X = \disjointunion_{\sigma\in \Sigma_X} (F_{V\setminus\sigma})^\circ =
    \disjointunion_{V\setminus\sigma\in\Sigma_X} (F_\sigma)^\circ.
  \end{equation}
  The three decompositions \eqref{eq:simplex_full}, \eqref{eq:dual_complex}
  and \eqref{eq:PX} prove the claim.
\end{proof}

Then we have proved the following result.
\begin{theorem}
	Let $X$ be a manifold with embedded corners (all faces connected) such
        that $\Sigma_X$ is not the full simplex on the vertex set $V$ with
        embedding $\tilde X\to \R^{m-|V|}$ as above. Then there
        is a homeomorphism
        $$O_X\cong\R^{m-|V|+1}\times(|\Delta(V)|\setminus |\Sigma_X^\vee|),$$
        where $|\Sigma_X^\vee|$ denotes the geometric realization of the dual of $\Sigma_X$.
      \end{theorem}
      
To prove the main result of this note we need to recall the Spanier-Whitehead duality theorem, for a proof see \cite{adams}.

\begin{theorem}
	Let $E^*$ be a generalized cohomology theory with dual homology theory
        $E_*$. let $|K|\subsetneq|\partial\Delta(V)|$ be the geometric
        realization of a simplicial complex with vertex set $V$ consisting of
        $|V|$ elements, so that $\dim(\boundary|\Delta(V)|)= |V|-2$. Then
        there is a canonical isomorphism
        $$\widetilde{E}^r(|K|)\cong\widetilde{E}_{|V|-3-r}(|K^\vee|).$$  
\end{theorem}

Applying this to K-theory and to $|\Sigma_X|$ we obtain the following theorem.
      
\begin{theorem}
	Let $X$ be a manifold with embedded corners with associated simplicial
        complex $\Sigma_X\ne \Delta(V)$ on the vertex set $V$. Then we have a canonical isomorphism
	$$K_*(C^*(\K_b(X)))\to \widetilde{K}_{-*-1}(|\Sigma_X|).$$
\end{theorem}
\begin{proof}
  We already know by Theorem \ref{CLV} that
  $$K_*(C^*(\K_b(X)))\to K^{m+*}(O_X).$$
  On the other hand, by definition
  $$K^*(O_X)\cong \widetilde{K}^*((O_X)^+),$$
  where $(O_X)^+$ is the one-point compactification. But
  $$O_X^+\cong S^{m-|V|+1}\wedge P_X^+\cong
  S^{m-|V|+1}\wedge\left(|\Delta(V)|/(|\Delta(V)|\setminus
    P_X)\right)=S^{m-|V|+1}\wedge\left(|\Delta(V)|/(|\Sigma_X^\vee|)\right).$$

  We therefore get
  \begin{equation*}
    \begin{split}
      K^{m+*}(\tilde O_X) &= \widetilde K^{m+*}((\tilde O_X)^+) \\
                        & \cong \widetilde
                        K^{*+|V|-1}(|\Delta(V)|/|\Sigma_X^\vee|)\\
                        &\cong \widetilde K^{*+|V|-2}(|\Sigma_X^\vee|)\\
                        &\cong \widetilde K_{-*-1}(|\Sigma_X|).
    \end{split}
  \end{equation*}
  Here, the first isomorphism is the definition of compactly supported
  K-theory, the second is the suspension isomorphism, the third is the
  boundary map in the long exact pair sequence, using that $|\Delta(V)|$ is
  contractible and the last is Spanier-Whitehead duality.
\end{proof}

We get the corresponding result for conormal homology $H_*^{cn}(-)$ defined in \cite{CL}. To prove this, one
could construct a direct correspondence between the chain complex which
defines conormal homology and the simplicial chain complex of $\Sigma_X$. We
use the shortcut that in \cite[Corollary 4.2]{CLV} it is already established that $H^{cn}_*(X)\cong
H^{m-*}(O_X)$. Combined with the argument above, applied to ordinary
(co)homology instead of K-theory, we obtain the final result of this note.

\begin{theorem}
Let $X$ be a manifold with embedded corners (all faces connected) with
associated simplicial complex $\Sigma_X$. Then we have an isomorphism 
$$H^{cn}_*(X)\to \widetilde{H}_{*-1}(|\Sigma_X|).$$
\end{theorem}

In particular, Theorem \ref{realization} implies that for every finite
abelian group there are examples such 
conormal homology contains that torsion group.

\bibliographystyle{plain}
\bibliography{simplicial} 
\end{document}